\documentclass[a4paper,12pt]{amsart}
\usepackage[utf8]{inputenc}
\usepackage[T1]{fontenc}
\usepackage{fullpage}
\usepackage[english]{babel}
\usepackage{amsthm,enumerate,amsmath,amscd,amssymb,latexsym,bbm,amsfonts}
\usepackage[dvips]{graphicx}

\usepackage{color}
\usepackage{pxfonts}
\usepackage{textcomp}
\usepackage{verbatim}
\usepackage{srcltx}
\theoremstyle{definition}
\newtheorem{definition}{Definition}[section]

\newtheorem{example}{Example}
\newtheorem{remark}{Remark}%[definition]
\theoremstyle{plain}
\newtheorem{lemma}[definition]{Lemma}
\newtheorem{theorem}[definition]{Theorem}

\newcommand*{\bR}{\ensuremath{\mathbb{R}}}

\newcommand*{\bN}{\ensuremath{\mathbb{N}}}
\newcommand*{\bB}{\ensuremath{\mathbb{B}}}

\newcommand{\N}{\mathbbm{N}}
\newcommand{\R}{\mathbbm{R}}

\newcommand{\B}{\mathcal{B}}
\newcommand{\Hau}{\mathcal{H}}
\newcommand{\abs}[1]{\left\vert #1 \right\vert}
\newcommand{\norm}[1]{\Arrowvert #1 \Arrowvert}
\DeclareMathOperator{\diam}{diam}

\DeclareMathOperator{\dist}{dist}
\def\Xint#1{\mathchoice
   {\XXint\displaystyle\textstyle{#1}}%
   {\XXint\textstyle\scriptstyle{#1}}%
   {\XXint\scriptstyle\scriptscriptstyle{#1}}%
   {\XXint\scriptscriptstyle\scriptscriptstyle{#1}}%
   \!\int}
\def\XXint#1#2#3{{\setbox0=\hbox{$#1{#2#3}{\int}$}
     \vcenter{\hbox{$#2#3$}}\kern-.5\wd0}}
\newcommand{\meanint}{\Xint-}

\title{Boundary blow up under Sobolev mappings}
\author{Aapo Kauranen and Pekka Koskela}
\begin{document}

\begin{abstract}
We prove that for mappings in $W^{1,n}(\bB^n, \R^m),$ continuous  up to the
boundary, with
modulus of continuity satisfying a certain divergence
condition, the
image of the boundary of the unit ball has zero $n$-Hausdorff measure.
For H\"older continuous mappings we also prove an essentially sharp generalized
Hausdorff dimension estimate.
\end{abstract}

\footnotetext{
\endgraf The authors were partially supported by the Academy of Finland grants
131477 and 263850.}

\maketitle

\section{Introduction}
Throughout the paper $\bB^n$ denotes the unit ball in $\R^n$ and
$W^{1,n}(\bB^n,\R^m)$ is the Sobolev space of $L^n(\bB^n,\R^m)$-functions
$f:\bB^n\rightarrow \R^m$ with  weak first order derivatives  in $L^n(\bB^n)$.

If $f:\bB^2\rightarrow \Omega\subset \R^2$ is a conformal mapping, then
the boundary of $\Omega$ can have positive Lebesgue measure even if $f$
extends continuously up to the boundary of the disk. If one requires  more,
for example uniform H\"older continuity, then $\partial\Omega$  is necessarily
of Lebesgue  measure zero. In fact, Jones and Makarov proved in \cite{JoMa95}
that
$\partial\Omega$ has
measure zero if  $f$ satisfies $\abs{f(z)-f(w)}\leq \psi(\abs{z-w})$ in $\bB^2$
for
$\psi:\left[0,\infty\right)\rightarrow\left[0,\infty\right)$ with
\begin{equation}
 \label{intehto}
\int_0 \abs{\frac{\log \psi(t)}{\log t}}^2\frac{\mathrm{d}t}{t}=\infty.
\end{equation}
This condition is very sharp: if the integral in \eqref{intehto} converges then
\cite{JoMa95} provides us with a simply connected domain $\Omega$ and a
conformal mapping $f:\bB^2\rightarrow\Omega$ so that the boundary of $\Omega$
has positive Lebesque measure and $f$ has the modulus of continuity $\psi$.

Our first result gives a surprisingly general extension of the conformal
setting; notice that each uniformly continuous conformal mapping
$f:\bB^2\rightarrow \Omega$ belongs to $W^{1,2}(\bB^2, \R^2).$
\begin{theorem}
\label{thm2}
 Let $f\in W^{1,n}(\bB^n, \R^m)$ be a continuous mapping so that
\begin{equation}
\label{jvuusmoduli}
\abs{f(z)-f(w)}\leq \psi(\abs{z-w})
\end{equation}
 for all $z,w\in \bar{\bB}^n,$ where $\psi:(0,\infty)\rightarrow(0,\infty)$ is
an allowable modulus of continuity with

\begin{equation}
\label{divehto2}
\int_0 \left|\frac{\log \psi(t)}{\log t}\right|^{n}\frac{dt}{t}=\infty.
\end{equation}

 Then  $\Hau^n(f(\partial \bB^n))=0.$
\end{theorem}

Above, $\Hau^n(A)$ denotes the $n$-dimensional Hausdorff measure of a set $A.$

For the definition of an allowable modulus of continuity see Section
\ref{preliminaries} below. For example, $\psi(t)=Ct^{\gamma},$ $0<\gamma<1,$ and

\begin{equation*}
 \psi_{l,s}(t)=\exp\left(
-C \frac{\left(
\log
\frac{C_l}{t}\right)^{\frac{n-1}{n}}}{\left(\log^{(l)}
\frac{C_l}{t}\right)^{\frac{s}{n}}\left(\prod\limits_{k=2}^{l-1}\log^{(k)}
\frac{C_l}{t}\right)^{\frac{1}{n}}}
\right)
\end{equation*}
 are allowable for all integers $l\geq 2$ and all $s>0$. Notice that
$\psi_{l,s}$ satisfies \eqref{divehto2}
if and only if $s\leq1.$ Here
$C>0,$   $\log^{(k)} t$ is the $k$-times iterated logarithm and  $C_l$ is 
any constant with $\log^{(l)}\frac{C_l}{2}\geq1$. 

Let us look at the special case $n=m=2$ of Theorem \ref{thm2} in the
H\"older continuous setting: $\psi(t)=Ct^{\gamma},$ where $0<\gamma\leq1.$
Consider a space filling (Peano) curve, i.e. a continuous mapping $g$ from
the unit circle onto a square. In the standard construction, $g$ is H\"older
continuous with exponent $\gamma=1/2.$ If one takes, say, the Poisson 
extension $f$ of $g$ to the unit disk, then $f$ is also H\"older continuous.
It is easy to check by hand that the partial derivatives of $f$ do not belong to
$L^2(\bB^2).$ By  Theorem \ref{thm2} no H\"older continuous (or even continuous
with control
function satisfying \eqref{divehto2}) extension $f$ 
of a space filling curve can satisfy $|Df|\in L^2(\bB^2).$

In the H\"older continuous case,  Jones
and Makarov actually proved that the Hausdorff dimension of $f(\partial\bB^2)$
is strictly less than two for conformal $f$. Contrary to the area zero
results, this dimension estimate is truely conformal in the following sense.

\begin{example}
 Let $p>1.$ There exists a locally H\"older continuous homeomorphism
$f:\R^2\rightarrow\R^2$ with $f\in W^{1,2}_{loc}(\bB^2,\R^2)$, which maps
$\partial\bB^2$ onto a set of positive $\Hau^g$-measure, for the gauge function
$g(t)=t^2(\log \frac{1}{t})^{p}.$
\end{example}

Here $\Hau^g$ denotes the generalized Hausdorff measure
with the function $g(t)$ as the dimension gauge. The precise definitions are
given
in Section \ref{preliminaries}.
Our second result gives a rather optimal positive result.
\begin{theorem}
\label{thm1}
 Let $f\in W^{1,n}(\bB^n, \R^m)$ and fix $0<\gamma\leq1$ and $C_0>0.$ 
If f satisfies
$$
\abs{f(z)-f(w)}\leq C_0 \abs{z-w}^{\gamma}
$$
for all $z,w \in \bB^n,$ then
$\Hau^g(f(\partial \bB^n))=0,$ for the gauge function $g(t)=t^n\log\frac{1}{t}.$
\end{theorem}

Jones and Makarov proved their result via harmonic measure and hence this
technique does not work in the setting of Theorem \ref{thm2}. An alternate
approach, relying on the conformal invariance of (quasi)hyperbolic metric, was
given in Koskela-Rohde \cite{KoRo97}, see \cite{Ni06}. Furthermore, Mal\'{y} and
Martio \cite{MM95} established Theorem \ref{thm2} in the H\"older continuous
case via a technique that we have not been able to push further.

Let us briefly describe the idea of the proof of Theorem \ref{thm2}. We
consider a Whitney decomposition of $\bB^n$ and assign each $Q\in \mathcal{W}$
a vector $f_Q\in \R^m$ and a radius $r_Q.$ The vector $f_Q$ will simply be the
``average'' of $f$ over $Q$ and $r_Q$ the maximum of $\abs{f_Q-f_{\hat{Q}}}$
over all neighbors of $Q.$ Then the $n$-integrability of the weak derivatives of
$f$ guarantees, via the Poincar\'e inequality, that the sequence
$\{r_Q\}_{Q\in \mathcal{W}}$ belongs to $l^n.$ We realize $f(\partial\bB^n)$ as
(a part of) the closure of $\{f_Q\}_{Q\in \mathcal{W}}$ in $\R^m.$ Those
$f(\omega),$ $\omega\in \partial\bB^n,$ for which one can find a sequence of
$Q\in\mathcal{W}$ with $\abs{f_Q-f(\omega)}\lesssim r_Q$ are easily handled.
For the remaining $\omega \in \partial\bB^n,$ we modify our centers $f_Q$ and
radii $r_Q,$ still retaining the $l^n$-condition, so that suitably blown up
balls cover these points sufficiently many times. This is where the
non-integrability condition \eqref{divehto2} kicks in. One cannot
fully follow the above idea, and our proof below is more complicated.

Our approach is flexible and applies to many related problems. In order to
avoid extra technicalities, we do not record such applications here. Let us
simply mention that the dimension gap phenomenon from \cite{HeKoNi11} can be
shown to extend from conformal mappings to general Sobolev mappings 
\cite{KoZa12}.

\section{Preliminaries}
\label{preliminaries}
Let us first agree on some basic notation. Given a number $a>0$, we write
$\lfloor a\rfloor$ for the largest integer less or equal to $a$. Similarly,
$\lceil a\rceil$ is the smallest integer greater or equal to $a$. If $A$ is a
finite set set, $\sharp A$ is the number of elements in $A$. 
If $A\subset\bR^n$ has finite and strictly positive Lebesgue measure and
$f\colon\bR^n\to\bR$ is a Lebesgue integrable
function, we denote the average $\frac{1}{\abs{A}}\int_A f$ of $f$ over the set
$A$ by $\meanint_A f$ or $f_A$, where $|A|$ is the $n$\nobreakdash-dimensional
Lebesgue measure of the set $A$. For $f:\R^n \rightarrow \R^m$ $f_A$ is then
defined via the component functions of $f.$  Given a point $x\in\bR^n$ and a
non-negative number $r$, $B(x,r)$ denotes the open ball with centre $x$ and
radius $r$ and $Q(x,r)$ denotes the cube
$\{y\in\bR^n\colon\max\{|x_i-y_i|\}_{i=1,2,\ldots,n}\leq r\}$. If $B=B(x,r)$ is
a ball and $a$ is a positive number, the notation $aB$ stands for the ball
$B(x,ar)$. We denote  the radius of a ball $B$ by $r(B)$. If we write
$L=L(\cdot)$, we mean that the number $L>0$ depends on the parameters listed in
the parentheses. Finally, $C$ denotes a positive constant, which may depend
only on 
$n$ and $m$, the dimensions of the domain space and the image space,  and may
differ from occurrence to
occurrence.

We write $\mathcal{H}^h(A)$ for the generalized Hausdorff
measure of a set $A\subset\bR^n$, given by
$$
\mathcal{H}^h(A)=\lim_{\delta\to0}\mathcal{H}^h_\delta(A),
$$
where
$$
\mathcal{H}^h_\delta(A)=\inf\Bigl\{\sum \limits_{i=1}^{\infty}
h(\diam U_i)\colon A\subset\bigcup\limits_{i=1}^{\infty} U_i,
\diam U_i\leq\delta\Bigr\}
$$
and $h$ is a dimension gauge (a non-decreasing function with
$\lim_{t\to0+}h(t)=h(0)=0$ and with $h(t)>0$ for all $t>0$). If
$h(t)=t^a$ for some $a\geq0$, we simply write
$\mathcal{H}^a$ for $\mathcal{H}^{h}$ and call it the
$a$-dimensional Hausdorff measure.

We need also a generalized weighted
Hausdorff content of a set $A\subset \R^n,$ given by
$$
\lambda_{\infty}^h(A)= \inf\left\lbrace \sum_{i=1}^{\infty} c_i h(\diam U_i)
\colon \chi_A(x)\leq \sum_{i=1} ^{\infty} c_i \chi_{U_i}(x),\, \forall x\in \R^n
\right\rbrace.
$$
Also here $h$ is  a gauge function. A sequence of pairs $(c_i,
U_i)_{i=1}^{\infty}$, where  $c_i\geq0$ and $U_i\subset \R^n,$ that satisfies
$\chi_A(x)\leq \sum_{i=1}^{\infty} c_i \chi_{U_i}(x),$ is called a
weighted cover of the set $A$. Again, we write $\lambda^h_ {\infty}=\lambda^a_
{\infty},$ if $h(t)=t^a.$

\begin{lemma}
\label{Howroyd}
 Let $E\subset \R^n$ be bounded. Let $h$ be a continuous gauge function
with
$h(2t)\leq c h(t)$ for some $c >0.$ Then
$\Hau_{\infty}^{h}(E)\leq c \lambda_{\infty}^h(E)$.
\end{lemma}
\begin{proof}
 The lemma follows from Corollary 8.2 and the proof of  Theorem 9.7 of
\cite{HowroydPhD} (see also \cite[2.10.24]{FedererGMT}).
\end{proof}

 Recall that for each open subset $U$ of $\bR^n$ there exist a Whitney
decomposition $U=\bigcup\limits_{i=1}^\infty Q_i$, where $Q_i$ are cubes
with mutually parallel sides, pairwise disjoint interiors and each of edge
length
$2^k$ for some integer $k$, such that the relation
\begin{equation}
\label{kuutiokoko}
\frac{1}{4}\leq\frac{\diam Q_i}{\dist(Q_i,\partial\Omega)}\leq1
\end{equation}
holds for all $i=1,2,\ldots$. We write $Q_1\backsim Q_2$, if the Whitney cubes
$Q_1\neq Q_2$ share at least one point (the so-called neighbor cubes). We have
$$
\frac{1}{4}\leq\frac{\diam Q}{\diam\tilde{Q}}\leq4,
$$
once $Q\backsim\tilde{Q}$. Therefore, the total number $\sharp\{\tilde{Q}\colon
\tilde{Q}\backsim Q\}$ of all neighbours of a fixed cube $Q$ does not exceed
$C$.
See~\cite{SteinSingInt} for details.

Let $\omega\in\partial\bB^n$. By $(Q_j(\omega))_{j=1}^\infty$, we mean the
sequence of all Whitney cubes in a fixed Whitney decomposition of $\bB^n$,
intersecting the radius $[0,\omega]$. This sequence starts with a central cube
and tends to $\omega$. 
For a point $x\in[0,\omega],$ we denote the number of Whitney cubes
intersecting the segment $[0,x]$ by $\sharp q(0,x).$ It is easy to see that
\begin{equation}\label{comparable}
c_1\leq\frac{\sharp q(0,x)}{\log\frac{1}{1-|x|}}\leq c_2,
\end{equation}
whenever $\sharp q(0,x)>c_3,$ where $c_i>0,$ $i=1,2,3$ are constants that may
depend on $n$.

Finally we define the allowable moduli of continuity.
\begin{definition}
 A continuously differentiable increasing bijection
$\psi:(0,\infty)\rightarrow(0,\infty)$ is an \emph{allowable modulus of
continuity} if there exists  $t_0<1$ and $\beta >0$ such that for every
$t\leq t_0$ the following conditions hold:
\begin{equation}
\label{logom1}
\log\frac{1}{\psi^{-1}(t)}\text{ is differentiable and
}\frac{(\psi^{-1})^{\prime}(t)}{\psi^{-1}(t)}t\text{ is a decreasing
function;}
\end{equation}
\begin{equation}
\label{logom2}
 \log\frac{1}{\psi^{-1}(t)}\leq\beta \log\frac{1}{\psi^{-1}(\sqrt{t})};
\end{equation}
\begin{equation}
 \label{monotpsi}
\frac{(\log \psi(t))^{\prime} t \log t}{\log \psi(t)} \text{ is a monotone
function.}
\end{equation}
\end{definition}

\begin{remark}
 \begin{enumerate}
  \item[i)]
  One could replace the monotonicity conditions in \eqref{logom1} and
\eqref{monotpsi} with a \emph{pseudomonotonicity} condition (e.g. there exists
a constant $C>0$ such that $u(t)\leq Cu(s)$ if $t\leq s$). This would only
affect the constants in the proofs.
  \item[ii)] 
The conditions \eqref{logom1} and \eqref{logom2} mean that the function
$\log\frac{1}{\psi^{-1}(t)}$ is a function of logarithmic type in the sense of 
\cite[Definition 4.2.]{Ni06}.
 \end{enumerate}
\end{remark}

\section{Proofs}
\begin{proof}
We may assume that $m,n\geq 2.$ Let $f\in W^{1,n}(\bB^n,\R^m)$ and  $\psi$ be
as in the statement of Theorem \ref{thm2}. 
Denote $\psi^{-1}(t)$  by $u(t).$ It follows from our assumptions
\eqref{divehto2}, \eqref{logom1}, \eqref{logom2},
\eqref{monotpsi} and \cite[Remark
5.3.]{Ni06} that
\begin{equation}
\label{divehto}
\int_0 \left(\frac{u(t)}{u'(t)}\right)^{n-1}\frac{dt}{t^n}=\infty.
\end{equation}

We define
$\alpha(t)=\frac{u(t)}{u'(t)}$ and
$\lambda(k)=\frac{2^{-k}}{\alpha(2^{-k})}$ for $k\in\N$. By \eqref{logom1},
$\lambda$ is
increasing for large $k$. For simplicity we assume $\lambda$ to be increasing.

Let $\mathcal W$ be a fixed Whitney decomposition of $\bB^n$. For each cube
$Q\in\mathcal W$, we define a corresponding centre $f_Q$ and a corresponding
radius \mbox{$r_Q=\max\{|f_Q-f_{\tilde{Q}}|\colon Q\backsim\tilde{Q}\}$}, which
determine a family of balls on the image side: $\mathcal
B=\{B(f_{Q},r_{Q})\colon Q\in\mathcal W, r_{Q}>0\}$. 
Note that some balls in $\mathcal B$ may coincide, the simplest way to act in
such a situation is to treat them as different balls  for certainty (we may
identify each ball in $\mathcal B$, with
$(Q,B(f_{Q},r_{Q}))$, then different Whitney cubes on
the pre-image side generate different pairs), however, identifying such balls
would cause no problem either.

We assign two new weighted collections of balls to each ball in $\mathcal B$.
Given  $B=B(x,r) \in \mathcal B$, we define concentric
subballs $S_i(B)=B(x,r/2^i)$
for all $i\in \N$ and assign the weight $w_{S_i(B)}=2^i$ to each $S_i(B).$
We set  $\mathcal{S}_B=\{S_i(B)\colon i\in \N\}$. Then
$$
\sum_{B'\in \mathcal{S}_B} w_{B'}r(B')^n=\sum\limits_{i=1}^\infty
w_{S_i(B)}r(S_i(B))^n = \sum\limits_{i=1}^{\infty} 2^i
\frac{r(B)^n}{2^{ni}}\leq r(B)^n.
$$

The second collection is defined in a similar way. If $B=B(x,r)$ is a
ball in $\mathcal B$, we choose the smallest number $k_0(B)\in\bN$, such that
$2^{-k_0(B)}\leq r$. Next, for each $k=k_0(B),k_0(B)+1,\ldots$, we choose
$R_k(B)=B(x,\alpha(2^{-k}))$ and set $\mathcal R_B=\{R_k(B)\colon
k=k_0(B),k_0(B)+1,\ldots\}$. The weights we assign this time are
$w_{R_k(B)}=\lambda(k)$
for all $k=k_0(B),k_0(B)+1,\ldots$.
Similarly to above:
\begin{align*}
\sum_{B'\in \mathcal{R}_B} w_{B'}r(B')^n&=\sum\limits_{k=k_0(B)}^\infty
w_{R_k(B)}r(R_k(B))^n= \sum\limits_{k=k_0(B)}^\infty
\left(\alpha(2^{-k})\right)^n \lambda(k)\leq
\sum\limits_{k=k_0(B)}^\infty
\left(\alpha(2^{-k})\right)^n \frac{\lambda(k)^n}{\lambda(0)^{n-1}}
\\&=\frac{1}{\lambda(0)^{n-1}}\sum\limits_{
k=k_0(B)}^\infty2^{-nk}\leq \frac{2}{\lambda(0)^{n-1}}\cdot 2^{-nk_0(B)}\leq
 \frac{2}{\lambda(0)^{n-1}}r(B)^n.
\end{align*}

Finally, we define our weighted collection of balls by setting $\mathcal
F=\bigcup\limits_{B\in\mathcal B}\Bigl(\mathcal{S}_B\cup\mathcal
R_B\Bigr)$. Again, some of the balls in the united families may coincide;
however, we treat them as "different" balls. Distinguishing them is,
again, not difficult.

Let us now estimate the weighted sum of the $n$th powers of the radii of the
balls in $\mathcal F$. Let 
$N(Q)=Q\cup\bigcup\limits_{\tilde{Q}\backsim
Q}\tilde{Q}$ be the union of all neighbors of a cube $Q\in\mathcal
W$. For neighboring cubes $Q$ and $Q',$ we obtain, via the H\"older and
Poincar\'{e} inequalities, that
\begin{align*}
\abs{f_Q-f_ {Q'}}&\leq \meanint_{Q} \abs{f-f_{N(Q)}}+\meanint_{Q'}
\abs{f-f_{N(Q)}}\leq C \meanint_{N(Q)} \abs{f-f_{N(Q)}}\leq C
\left(\,\meanint_{N(Q)}
\abs{f-f_{N(Q)}}^n\right)^{1/n}\\&\leq C \left(\,\int_{N(Q)}
\abs{Df}^n\right)^{1/n}.
\end{align*}
Hence, we have the
 estimate
$$
r_Q^n=\max\{|f_Q-f_{\tilde{Q}}|^n\colon Q\backsim\tilde{Q}\}\leq
C\int_{N(Q)}|Df|^n
$$
for each $Q\in\mathcal W$ and some constant $C>0$. Next, using the fact
that the inequality $\sum\limits_{Q\in\mathcal
W}\chi_{N(Q)}(y)\leq C$ holds for every $y\in\bR^n$, we estimate
\begin{align}\label{area}
\notag\sum\limits_{B\in\mathcal
F} w_Br(B)^n&\leq\,C(\lambda(0))\sum\limits_{B\in\mathcal
B}r(B)^n=C(\lambda(0))\sum\limits_{Q\in\mathcal
W}r_Q^n\leq
C(\lambda(0))\sum\limits_{Q\in\mathcal W}\int_{N(Q)}|Df|^n\\
&\leq
\,C_1\int_{\bigcup\limits_{Q\in\mathcal W}N(Q)}|Df|^n \leq C_1\int_{\bB^n}|Df|^n
< \infty,
\end{align}
where $C_1>0$ is some constant depending on $n,$ $m$ and $\lambda(0)$ only.

We may assume
that  there is at least one $Q\in \mathcal W$ with $r_Q>0;$
otherwise $f(\partial\bB^n)$ is a singleton. 
Let $\omega\in\partial\bB^n$. We  consider the radius
$[0,\omega]$ and the sequence $(Q_j(\omega))_{j=1}^{\infty}$. 
%We may assume
%that  there is at least one cube in the sequence with positive edge length;
%otherwise $f(\partial\bB^n)$ is a singleton.   
We fix a large
integer $l_0=l_0(\omega,f)\in\bN$ so that there are elements of the
sequence
$(f_{Q_j(\omega)})_{j=1}^{\infty}$ outside  $B(f(\omega),2^{-l_0+1})$,
if $(f_{Q_j(\omega)})_{j=1}^{\infty}$ contains at least one element different
from $f(\omega)$. If such an integer does not exist, there necessarily is
some $Q=Q_w\in \mathcal W$ with $f_Q=f(\omega)$ and $r_Q>0.$ 
In this case, 
%we
%we choose one ball cenred at
%$f(\omega)$ with edge length $r_(\omega)>0$ and 
we choose 
$l_0=l_0(\omega,f)\in\bN$ so that $2^{-l_0}<r_{Q_\omega}$. In both cases we also
require  that $2^{-l_0+1}<t_0.$  This allows us to use the properties
\eqref{logom1} and \eqref{logom2}.

% To use the
%properties \eqref{logom1} and \eqref{logom2} we also require  that
%$2^{-l_0+1}<t_0.$

For the purposes of our "porosity argument", we would like to make the number
$l_0$ independent of the point~$\omega$. This is done by considering
the
decomposition
$$
\partial\bB^n=\bigcup\limits_{l\in\bN}E_l,\text{ where
}E_l=\{\omega\in\partial\bB^n \colon
l_0(\omega,f)\leq l\}.
$$
Setting $F_l=f(E_l)$, we then have
$f(\partial\bB^n)=\bigcup\limits_{i\in\bN}F_l$.

Let us fix  $l_0\in\bN$. Our aim is to prove that $\mathcal
H^n_{\infty}(F_{l_0})=0$. 

Fix $x\in F_{l_0}.$ Take any $\omega\in E_{l_0},$ such that $x=f(\omega),$
and define the sequence of
concentric annuli $A_l(x)=B(x,2^{-l+1})\setminus B(x,2^{-l})$ with
$l=l_0,l_0+1,\ldots$. Next, we assign a suitable set $P_l(x)$ of cubes from
$\mathcal W$ to each annulus $A_{l}(x)$, $l=l_0,l_0+1,\ldots$. If
$f_{Q_j(\omega)}=x$ for all $j\in \N$, we put $P_l(x)=\{Q_\omega\}$ for
each $l\geq l_0$, where $Q_\omega$ is the cube defined earlier. Otherwise, all
the sets $P_l(x)$ with $l\geq l_0$ consist of elements from
$(Q_j(\omega))_{j=1}^{\infty}$: if an annulus $A_{l}(x)$ with some $l\geq l_0$,
contains no centres from $(f_{Q_j(\omega)})_{j=1}^{\infty}$, we define
$P_l(x)=\{Q_m(\omega)\}$, where an integer $m\in \N$ is chosen so that
$f_{Q_{m-1}(\omega)}\not\in B(x,2^{-l+1})$, but $f_{Q_m(\omega)}\in
B(x,2^{-l})$; if, in contrast, there is at least one centre $f_{Q_j(\omega)}$ in
$A_{l}(x)$, we take $P_l(x)=\{Q_k(\omega)\colon k=m_1,\ldots,m_2\}$,
where $m_1,m_2\in \N$ are such that $f_{Q_{m_1-1}(\omega)}\not\in
B(x,2^{-l+1})$, $f_{Q_{m_2+1}(\omega)}\in B(x,2^{-l})$ and $f_{Q_k(\omega)}\in
A_{l}(x)$ for all $k=m_1,\ldots,m_2$. Moreover, it is possible to choose the
sets $P_l(x)$ above so that the inequality $k_1\leq k_2$ is valid,
whenever $Q_{k_1}(\omega)\in P_{l_1}(x)$, $Q_{k_2}(\omega)\in P_{l_2}(x)$ and
$l_1<l_2$.
\begin{comment}
Fix $\omega\in E_{l_0},$ set $x=f(\omega),$ and define the sequence of
concentric annuli $A_l(x)=B(x,2^{-l+1})\setminus B(x,2^{-l})$ with
$l=l_0,l_0+1,\ldots$. Next, we assign a suitable set $P_l(x)$ of cubes from
$\mathcal W$ to each annulus $A_{l}(x)$, $l=l_0,l_0+1,\ldots$. If
$f_{Q_j(\omega)}=x$ for all $j\in \N$, we put $P_l(x)=\{Q_\omega\}$ for
each $l\geq l_0$, where $Q_\omega$ is the cube defined earlier. Otherwise, all
the sets $P_l(x)$ with $l\geq l_0$ consist of elements from
$(Q_j(\omega))_{j=1}^{\infty}$: if an annulus $A_{l}(x)$ with some $l\geq l_0$,
contains no centres from $(f_{Q_j(\omega)})_{j=1}^{\infty}$, we define
$P_l(x)=\{Q_m(\omega)\}$, where an integer $m\in \N$ is chosen so that
$f_{Q_{m-1}(\omega)}\not\in B(x,2^{-l+1})$, but $f_{Q_m(\omega)}\in
B(x,2^{-l})$; if, in contrast, there is at least one centre $f_{Q_j(\omega)}$ in
$A_{l}(x)$, we take $P_l(x)=\{Q_k(\omega)\colon k=m_1,\ldots,m_2\}$,
where $m_1,m_2\in \N$ are such that $f_{Q_{m_1-1}(\omega)}\not\in
B(x,2^{-l+1})$, $f_{Q_{m_2+1}(\omega)}\in B(x,2^{-l})$ and $f_{Q_k(\omega)}\in
A_{l}(x)$ for all $k=m_1,\ldots,m_2$. Moreover, it is possible to choose the
sets $P_l(x)$ above so that the inequality $k_1\leq k_2$ is valid,
whenever $Q_{k_1}(\omega)\in P_{l_1}(x)$, $Q_{k_2}(\omega)\in P_{l_2}(x)$ and
$l_1<l_2$.
\end{comment}

Denoting
$$
\theta_l(x)=
\begin{cases}
1, &\text{ if }\sharp P_l(x)\leq \tilde{c}_0\lambda(l),\\
0, &\text{ otherwise,}
\end{cases}
$$
for $l\geq l_0$ and a constant $\tilde{c}_0>\lambda^{-1}(0)$, which we will
specify later,
we would like to prove that there exists an integer $l_1\geq 2l_0$, such that
\begin{equation}\label{kolmaskaava}
\sum\limits_{k=l_0}^{l}\theta_k(x)\geq\frac{l}{2}
\end{equation}
for each $l\geq l_1$. In other words, at least half of the annuli do not contain
too many centres from $(f_{Q_j(\omega)})_{j= 1}^{\infty}$. There is nothing to
prove, if $f_{Q_j(\omega)}=x$ for all $j\in \N$; otherwise, the proof is
by contradiction.

Let us assume that~\eqref{kolmaskaava} does not hold for some $l\geq2l_0$. Take
the smallest number $J\in \N$ such that $f_{Q_j(\omega)}\in
B(x,2^{-l})$ for all $j>J$ and let
$\omega^\prime\in[0,\omega]$ be the point of
$Q_{J}(\omega)\cap[0,\omega],$ which is the closest to
$\omega$. Now, the assumption on the continuity of $f$
and the properties of our Whitney decomposition imply
\begin{align*}
2^{-l}\leq|f_{Q_{J}(\omega)}-x|=|f_{Q_{J}(\omega)}
-f(\omega)|\leq \meanint_{Q_{J}}\abs{f(y)-f(\omega)}dy\leq \psi(2
(1-\abs{\omega^{\prime}})).
\end{align*}
That is,
$$
\frac{u(2^{-l})}{2}\leq 1-\abs{\omega^{\prime}}.
$$
Next, we connect this estimate to the number of Whitney cubes that precede
$Q_{J}$ in $(Q_j(\omega))_{i=1}^{\infty}$. 

Using~\eqref{comparable}, we observe that
$$
\log \frac{2}{u(2^{-l})}\geq \log
\frac{1}{1-\abs{\omega^\prime}}\geq \frac{1}{c_2} \sharp q(0,\omega^\prime).
$$

In the calculation above, we may have to adjust the choice of $l_0$ to ensure
$\sharp q(0,\omega^\prime)>c_3$ (see \eqref{comparable}). Finally, we obtain a
lower bound for $\sharp q(0,\omega^\prime)$,
using the assumption that we have at least $\lfloor l/2\rfloor-l_0+2$ annuli
$A_k(x)$ with $\theta_k(x)=0$. We notice that the sets $P_k(x)$ with
$\theta_k(x)=0$ contain different cubes for different $k$'s, and, if $k\leq
l,$
then the cubes in $P_k(x)$ precede $Q_{J}(\omega)$ in
$(Q_j(\omega))_{j=1}^{\infty}$. We have

\begin{align*}
c_2 \log\frac{2}{u(2^{-l})}&\geq \sharp
q(0,\omega^\prime)\geq\sum_{\substack{k=l_0,\ldots,l\\\theta_k(x)=0}}\sharp
P_k(x)\geq
\sum\limits_{k=l_0}^{\lfloor l/2\rfloor+1}\tilde{c}_0\lambda(k)
\geq \tilde{c}_0\sum\limits_{k=l_0}^{\lfloor
l/2\rfloor+1}\frac{2^{-k}u'(2^{-k})}{u(2^{-k})}\\&\geq\tilde{c}_0 \left(\log
\frac{1}{u(2^{-l/2})}-\log\frac{1}{u(2^{-l_0})}\right)\geq
\tilde{c}_0\beta^{-1} \log
\frac{1}{u(2^{-l})}-\tilde{c}_0\log\frac{1}{u(2^{-l_0})}.
\end{align*}
Choosing $\tilde{c}_0>c_2\beta$, this cannot hold when $l$ is large enough.
Thus, there is a number $l_1=l_1(\tilde{c}_0,l_0,u)$, such
that~\eqref{kolmaskaava} holds for all $l\geq l_1$.

Our next step is to prove that if $\theta_k(x)=1$ for some $k$ and
$P_k(x)=\{Q_1,\ldots,Q_m\}$, then it is
possible to find a collection of balls $\{B_1,\ldots,B_{m'}\}$ from the
families $\mathcal{S}_{B(f_{Q_i},r_{Q_i})}$ or
$\mathcal R_{B(f_{Q_i},r_{Q_i})}$, having radii at least
$const\cdot \alpha(2^{-k})$ and satisfying $\sum_{i=1}^{m'}w_{B_i}\geq
const\cdot \lambda(k)$. Moreover, we
choose different balls (in the sense mentioned above) for different~$k$'s.

Let us fix $k\geq l_0$ such that $\theta_k(x)=1$. Suppose first that
the annulus $A_{k}(x)$ contains no centres from $(f_{Q_j(\omega)})_{j=
1}^{\infty}$. Then the set $P_k(x)$ consists of a single cube $Q\in\mathcal
W$ with $f_Q\in B(x,2^{-k})$. The definitions of $r_Q$
and $l_0$
imply $r_Q>2^{-k}$, and hence $k\geq k_0(B(f_Q,\,r_Q))$. Thus, we may choose the
ball $R_k(B(f_Q,r_Q))$, which, by
definition,
has radius $\alpha(2^{-k})$ and weight $\lambda(k)$. In addition, the centre of
this ball lies in $B(x,2^{-k})$.

Assume now that the annulus $A_{k}(x)$ contains at least one of the centres from
$(f_{Q_j(\omega)})_{j= 1}^{\infty}.$ Then, we have by the definitions of
$P_k(x)$ and $r_{Q}$ that
$$
\sum_{Q\in P_k(x)}2r_Q \geq 2^{-k}.
$$
Since $\sharp P_k(x)\leq \tilde{c}_0\lambda(k),$ we observe that
$$
\sum_{\substack{Q\in P_k(x)\\ 2r_Q\geq \alpha(2^{-k})/2\tilde{c}_0}} 2r_Q \geq
\frac{2^{-k}}{2}.
$$
For each $Q\in P_k(x)$ with $2r_Q \geq \frac{\alpha(2^{-k})}{2\tilde{c}_0 },$
we
choose a number $n_Q\in \N$ so that
$$
2^{n_Q-1}\frac{\alpha(2^{-k})}{2\tilde{c}_0}\leq2r_Q<2^{n_Q}\frac{\alpha(2^{-k}
) } { 2\tilde { c } _0 }
$$
and pick a ball $\tilde{B}=S_{n_Q}(B(f_Q,r_Q))=B(f_Q,r_Q/2^{n_Q})\in \mathcal
S_{B(f_Q,r_Q)}$. By the definition of $S_i(B)$, we have
$w_{\tilde{B}}=2^{n_Q}$ and
$$
r(\tilde{B})=\frac{r_Q}{2^{n_Q}}\geq\frac{\alpha(2^{-k})}{8\tilde{c}_0}.
$$
For the sum of the weights $\sum_{Q}2^{n_Q}$ of all the balls obtained in such a
manner, we
observe that
$$
\frac{\alpha(2^{-k})}{2\tilde{c}_0}\sum_{\substack{Q\in
P_k(x)\\2r_Q\geq\alpha(2^{-k})
/2\tilde{c}_0}} 2^{n_Q}>
\sum_{\substack{Q\in P_k(x)\\2r_Q\geq\alpha(2^{-k})
/2\tilde{c}_0}}
2r_Q\geq \frac{2^{-k}}{2}.
$$
Hence we have a collection of balls $\{B_1,\ldots,B_m\}\subset \mathcal
F$ with weights sum $\sum_{i=1}^mw_{B_i}>\tilde{c}_0\lambda(k)$ and of radii
at least $\alpha(2^{-k})/8\tilde{c}_0$. Moreover, all these balls
have their centres in the annulus $A_k(x)$, and hence in the ball
$B(x,2^{-k+1}).$

We have proved that there exists a number $l_1=l_1(l_0,\tilde{c}_0)$, such that
for
each $\omega\in E_{l_0}$ and $l\geq l_1$, among the numbers $l_0,\ldots,l$,
there are at least $\lceil l/2\rceil$ integers $k\in\{l_0,\ldots,l\}$, such that
we are able to find a finite collection of balls $\{B_i\}_{i\in
I}\subset\mathcal F$ with weights sum $\sum_{i\in I}w_{B_i}$ at least
$\lambda(k)$ and of radii at least $\alpha(2^{-k})/8\tilde{c}_0$, so that the
centres of the
balls $B_i$, $i\in I$, lie in the ball $B(x,2^{-k+1})$. Here,
$\tilde{c}_0$ is a positive constant depending only on $\beta,$  $n$ and
$\lambda(0)$, and the
balls are different for a fixed $\omega$ and
different $k$'s.

Fix $l\geq l_1$. We modify our family $\mathcal F$ according to $l$.
If $B\in\mathcal F$ and there is $k\in\{l_0+1,\ldots,l\}$ such that
$\alpha(2^{-k})/8\tilde{c}_0\leq r(B)<\alpha(2^{-k+1})/8\tilde{c}_0$, we replace
$B$ with
the ball $\tilde{B}=\frac{\lambda(k)}{\lambda(l)}B$, and set
$w_{\tilde{B}}=(\lambda(l)/\lambda(k))^nw_B$. The radius of $\tilde{B}$
satisfies
$r(\tilde{B})\geq\frac{\lambda(k)}{\lambda(l)}
\alpha(2^{-k})/8\tilde{c}_0 = 2^{-k}/8\tilde{c}_0\lambda(l)$ and the equality
$w_{\tilde{B}}r(\tilde{B})^n=w_Br(B)^n$ holds.
Similarly, we replace a ball $B$ with $r(B)\geq\alpha(2^{-l_0})/8\tilde{c}_0$
with the
ball $\tilde{B}=\frac{\lambda(l_0)}{\lambda(l)}B$ and set
$w_{\tilde{B}}=(\lambda(l)/\lambda(l_0))^nw_B$.
Again, we have $r(\tilde{B})\geq2^{-l_0}/8\tilde{c}_0\lambda(l)$
and $w_{\tilde{B}}r(\tilde{B})^n=w_Br(B)^n$.
Finally, $\mathcal F_l$ is the collection  of balls obtained in this
manner from the balls in $\mathcal F$.  For this family of balls, we
notice (see~\eqref{area}) that
\begin{equation}
\label{taasarea}
\sum\limits_{B\in\mathcal
F_l}w_Br(B)^n\leq\sum\limits_{B\in\mathcal
F}w_Br(B)^n<\infty.
\end{equation}

If $\omega\in E_{l_0}$, $x=f(\omega)$ and $k\in\{{l_0},\ldots,l\}$ is
such that $\theta_k(x)=1$, then there is a collection $\{B_i\}_{i\in
I}\subset\mathcal F$ with the properties mentioned above. If a
ball $B_i$ with some $i\in I$ is replaced by a ball
$\tilde{B_i}=\frac{\lambda(k_i)}{\lambda(l)}B_i$, while creating $\mathcal
F_l$, we
necessarily have $k_i\leq k$. Therefore, the inequalities
\begin{comment}
$$
\sum\limits_{i\in I}w_{\tilde{B}_i}=\sum\limits_{i\in
I}\left(\frac{\lambda(l)}{\lambda(k_i)}\right)^nw_{B_i}\geq\left(\frac{
\lambda(l)}{\lambda(k) } \right)^n\sum\limits_ {
i\in I}w_{B_i} \geq\left(\frac{\lambda(l)}{\lambda(k) }
\right)^n\lambda(k)\geq\lambda(l)^n \frac{\lambda(k)\alpha(2^{-k})^n}{2^{-kn}}
$$
\end{comment}
$$
\sum\limits_{i\in I}w_{\tilde{B}_i}=\sum\limits_{i\in
I}\left(\frac{\lambda(l)}{\lambda(k_i)}\right)^nw_{B_i}\geq\left(\frac{
\lambda(l)}{\lambda(k) } \right)^n\sum\limits_ {
i\in I}w_{B_i} \geq\left(\frac{\lambda(l)}{\lambda(k) }
\right)^n\lambda(k)=\lambda(l)^n \frac{1}{\lambda(k)^{n-1}}
$$
and
$r(\tilde{B}_i)\geq 2^{-k_i}/8\tilde{c}_0\lambda(l)\geq2^{-k}/8\tilde{c}
_0\lambda(l)\, $ hold (by \eqref{logom1}, $\lambda$ is increasing).
Since, for each $i\in I$, the centre of a ball $\tilde{B}_i$ is contained in
$B(x,2^{-k+1})$, we have the inclusion
$x\in16\tilde{c}_0\lambda(l)\tilde{B}_i$. Hence we
observe that
$$
\sum\limits_{B\in\mathcal
F_l}w_B\chi_{16\tilde{c}_0\lambda(l)B}(y)\geq\sum_{\substack{
k=l_0 ,\ldots,
l\\\theta_k(y)=1
}} \lambda(l)^n
\frac{1}{\lambda(k)^{n-1}}\geq\frac{\lambda(l)^n}{4}\sum\limits_
{ k=l_1}^l\frac{1}{\lambda(k)^{n-1}}
\geq\frac{\lambda(l)^n}{4} G_l
$$
for each $y\in F_{l_0},$ where
$G_l=\sum\limits_{
k=l_1}^l\frac{1}{\lambda(k)^{n-1}}.$  That is,
$\left(\frac{4w_B}{\lambda(l)^n G_l},
16\tilde{c}_0\lambda(l)B\right)_{B\in \mathcal F_l}$ is a
weighted cover of
the set $F_{l_0}.$
We observe also that diameters of all balls in this cover are at least
$2^{-l}$. This information will be used in the proof of Theorem \ref{thm1}
below.

Finally, using the weighted cover obtained above and \eqref{taasarea}, we
estimate the 
weighted Hausdorff $n$-content $\lambda_{\infty}^n(F_{l_0})$:
\begin{align*}
\lambda_{\infty}^n(F_{l_0})&\leq \frac{4}{\lambda(l)^n G_l} \sum_{B\in
F_l} w_B(\diam 16\tilde{c}_0\lambda(l)B)^n\leq
\frac{4^{2n+1} \tilde{c}_0^n}{G_l} \sum_{B\in
F_l} w_B(\diam B)^n\\&\leq\frac{2^{5n+2}\tilde{c}_0^n}{G_l}\sum_{B\in
F_l} w_Br(B)^n\leq \frac{A}{G_l},
\end{align*}
where the constant $A$ depends on $\beta,$ $n,$ $m,$
$\norm{f}_{W^{1,n}(\B^n,\R^m)}$
and
$\lambda(0)$ but does not depend on $l_0$ or $l$.

Now,
Lemma~\ref{Howroyd} implies
$\Hau_{\infty}^n(F_{l_0})\leq \frac{CA}{G_l}.$ 
 Here $C$ depends only on the dimension $n.$
Now, we are done as soon as we can show that $G_l \rightarrow \infty$  as
$l\rightarrow \infty$. We have
$$
G_l=\sum_{k=l_1}^{l}\frac{1}{\lambda(k)^{n-1}}=
\sum_{k=l_1}^{l}\frac{u(2^{-k})^{n-1}}{2^{-k(n-1)}u'(2^{-k})^{n-1}}\geq
\int_{2^{-l}}^{2^{-l_1}}\left(\frac{u(t)}{u'(t)}\right)^{n-1}\frac{dt}{t^n}
$$
and the right hand side diverges as $l\rightarrow \infty$ by the assumptions on
the modulus of continuity.

\end{proof}

The proof of Theorem \ref{thm1} is similar to the proof of Theorem \ref{thm2}.
We only point out the required changes.

\begin{proof}[Proof of Theorem \ref{thm1}]
 Let $f$ be as in statement of the theorem. Our notation will be the same as in
previous proof. That is, $\alpha(t)=\gamma t$ and $\lambda(k)=\frac{1}{\gamma}.$

Fix a small $\varepsilon>0.$ Then there exists a $\delta>0$ such that
\begin{equation}
 \label{areaepsilon}
\int_{\bB^n\setminus B(0,1-\delta)} \abs{Df}^n\leq \varepsilon.
\end{equation}
Let $\mathcal W^{\,\delta}$ be the set of the cubes in $\mathcal W$ which are
contained in $\bB^n\setminus B(0,1-\delta)$ and whose all neighbour cubes are
also contained in $\bB^n\setminus B(0,1-\delta).$ We define our collection of
balls to be $\mathcal B^{\,\delta}=\{B(f_Q,r_Q)\colon Q\in\mathcal
W^{\,\delta} \}.$ Then, proceeding as in the previous proof, we define
$\mathcal F^{\,\delta}$ analogously to $\mathcal F$ and obtain the estimate (see
\eqref{area})
\begin{equation}
 \label{kokomassa}
\sum_{B\in \mathcal F^{\,\delta}}w_B r(B)^n\leq C_1 \varepsilon.
\end{equation}
Let $\omega \in \partial \bB^n.$ We define the number $l_0=l_0(\omega,f,\delta)$
as in previous proof, but instead of all cubes in
$(Q_j(\omega))_{j=1}^{\infty}$ we consider only those which are contained in
$\mathcal W^{\,\delta}.$
Again, we split $\partial \bB^n$ to sets $E_{l}=\{\omega\in \partial
\bB^n\colon l_0(\omega)\leq l\}$ and consider a fixed $f(E_{l'}).$
With the same method as earlier we find, for big $l$ a collection of balls
$\mathcal F
_{l}^{\,\delta}$ with weights such that
$(\frac{8w_B\gamma}{l-l_1},\frac{16\tilde{c}_0}{\gamma}B)_{B\in\mathcal
F_l^{\,\delta} }$ is a weighted cover of the set $f(E_{l'}),$ the radius
of the balls $\frac{16\tilde{c}_0}{\gamma}B$ is at least $2^{-l}$ and 
$$
\sum_{B\in \mathcal F_l^{\,\delta}} w_B r(B)^n\leq C_1 \varepsilon.
$$

We may assume that our $\varepsilon>0$ is so small that all balls in our
weighted cover have radius smaller than $\frac{1}{2}.$ With this weighted cover
we obtain 

\begin{align*}
\lambda^{g}_{\infty}(f(E_{l'}))&\leq \frac{4\gamma}{l-l_1}\sum_{B\in
\mathcal
F_l^{\,\delta}}w_B(\diam \frac{16\tilde{c}_0}{\gamma}B)^n\log \frac{1}{\diam
\frac{16\tilde{c}_0}{\gamma}B}\\&\leq \frac{4\gamma}{l-l_1}\sum_{B\in \mathcal
F_l^{\,\delta}}w_B(\diam \frac{16\tilde{c}_0}{\gamma}B)^n\log 2^l\leq
\frac{2^{2+5n}\tilde{c}_0^{n}}{\gamma^{n-1}}\frac{l}{l-l_1}\sum_{B\in \mathcal
F_l^{\,\delta}}w_Br(B)^n\leq\frac{2^{3+5n}\tilde{c}_0^{n}C_1}{\gamma^{n-1}}
\varepsilon.
\end{align*}
Here we assumed  $l$ to be so big that $\frac{l}{l-l_1}\leq 2.$
Lemma \ref{Howroyd} implies $\Hau^{g}_{\infty}(f(E_{l'}))\leq A
\varepsilon.$ Here $A$ depends on $\gamma,$  $n$ and $m$ but not on $l'$ or
$l$; 
therefore, we have $\Hau^{g}_{\infty}(f(\partial \bB^n))\leq A
\varepsilon,$ see \cite[Corollary 8.2]{HowroydPhD} or
\cite[2.10.22]{FedererGMT}. Letting $\varepsilon$ tend to zero gives
$\Hau^{g}_{\infty}(f(\partial \bB^n))=0,$ which implies
$\Hau^{g}(f(\partial \bB^n))=0.$
\end{proof}

\section{Example}
\label{esim}
We use the  notation $\norm{x}=\max\{\abs{x_1},\abs{x_2}\}.$ Let
$p>1/2.$
We will construct a mapping $f:\R^2\rightarrow \R^2$,
 with  $f\in W^{1,2}_{loc}(\R^2,\R^2),$ which is locally H\"older continuous
 and maps $\partial \bB^2$ to a set of positive
$\Hau^g$-measure, with $g(t)=t^2\left(\log\frac{1}{t}\right)^{2p}$.

The mapping is a composition of two locally H\"older continuous mappings.
The second mapping is defined in \cite[Prop. 5.1]{HerKo03}. It is a
homeomorphism
$h:\R^2\rightarrow \R^2,$  identity outside $[0,1]^2,$ and maps a
small Cantor set $\mathcal C\subset [0,1]^2$ to a
large Cantor set $\mathcal C' \subset [0,1]^2$ with positive $\Hau^g$-measure.
It was checked in
\cite{KZZ09} that this mapping belongs to $W^{1,2}_{loc}(\R^2,\R^2)$ if $p>1/2.$

Next, we elaborate on the construction of $h$ and prove that it is H\"older
continuous in $ [0,1]^2$. 
Let $\sigma<1/2.$
We use the notation $2r_k=\sigma^{k}$ and $2R_k=\frac12\sigma^{k-1}$
for $k\in\N.$ The set $\mathcal C$ is defined as follows:
In the first generation we have one square $Q_0=[0,1]^2$ with side length
$2r_0.$
We split this square to four squares $P_{1i},$ $i=1,2,3,4,$ of side length
$2R_1.$ We define the square $Q_{1i}$ to be the square of side length $2r_1$
centered at the center of $P_{1i}.$ Then $P_{1i}$ and $Q_{1i}$ generate
the frame $A_{1i}=P_{1i}\setminus
Q_{1i}.$ Next, we divide all squares $Q_{1i}$ to squares $P_{2j}$,
$j=1,\ldots,4^2.$ Then we define $Q_{2j}$ and $A_{2j}$ as in the first step.
We proceed inductively.
Thus, we have obtained for all $k\in\N$ sets $Q_{ki},$ $P_{ki}$ and $A_{ki},$
where $i=1,\ldots,2^{2k}$ and we have  $\mathcal C =\cap_{k}\cup_i Q_{ki}.$

The set $\mathcal C',$ and sets $Q'_{ki},$ $P'_{ki}$ and $A'_{ki},$ with
$k\in\N$ and $i=1,\ldots,2^{2k}$ are defined in the same way, using
$2r'_1=\frac12(\log 4)^{-p}$, $2R'_2=r'_1$ and $2r'_k=(\log4)^{-p}2^{-k}k^{-p}$
and
$2R'_k=(\log4)^{-p}2^{-k}(k-1)^{-p}$ for other $k\in \N.$

The mapping $h$ is defined so that it maps the frame $A_{ki}$ to the frame
$A'_{ki}$ via a ``radial'' stretching and is continuous in $[0,1]^2$. The
radial
stretching  which maps
$A=\{x:r_k\leq\norm{x}\leq R_k\}$ to $A'=\{x:r'_k\leq\norm{x}\leq R'_k\}$ is
$$
\rho(x)=(a\norm{x}+b)\frac{x}{\norm{x}}, \text{ where }
a=\frac{R'_k-r'_k}{R_k-r_k} \text{ and  } b=\frac{R_kr'_k-R'_kr_k}{R_k-r_k}.
$$

If $x,y\in A,$ then $\norm{x-y}\leq 2R_k=\frac12 \sigma^{k-1}$ and
$$
a\leq\frac{4\sigma}{1-2\sigma}(2\sigma)^{-k}\leq
C(\sigma)\sigma^{-(1-\beta)k}\leq C(\sigma)\norm{x-y}^{\,\beta-1},
$$ 
where $\beta=\log 2/\log \frac{1}{\sigma}.$ Similarly,
$$
\frac{\abs{b}}{\abs{r_k}}\leq \frac{4}{1-2\sigma}(2\sigma)^{-k}\leq
C(\sigma)\norm{x-y}^{\,\beta-1}.
$$
The mapping $\rho$ is H\"older continuous with exponent $\beta,$
$$
\norm{\rho(x)-\rho(y)}\leq C a
\norm{x-y}+2\frac{\abs{b}}{\abs{r_k}}\norm{x-y}\leq C(\sigma)\norm{x-y}^{\beta}.
$$

If $x\in A_{ki}$ and $y\in Q_{k+1,j}\subset P_{ki},$ then
$\norm{x-y}\geq R_{k+1}-r_{k+1}=C(\sigma)\sigma^{k}$ and
$\norm{h(x)-h(y)}\leq 2R'_k\leq 2^{-k}.$
These imply 
$$
\frac{\norm{h(x)-h(y)}}{\norm{x-y}^{\beta}}\leq C(\sigma).
$$
The $\beta$-H\"older continuity of $h$ easily follows from the continuity
estimates obtained above.

The first mapping $g:\R^2\rightarrow \R^2$ is a (locally H\"older continuous)
quasiconformal mapping for
which $\mathcal{C}\subset g(\partial \bB^2).$ 
%and which is locally H\"older continuous. 
Such a mapping was constructed in \cite{GeVa73}. 
%(see also \cite{Ge82} and \cite{Bi99}).

Finally, the composition $h\circ g:\R^2\rightarrow \R^2$ is a homeomorpism with
$h\circ g(\partial\bB^2)\supset\mathcal C'.$ Moreover, it is locally H\"older
continuous and $h\circ g \in W^{1,2}_{loc}(\R^2, \R^2)$ by quasiconformality of
$g$ and the change of variable formula.

\bibliography{../bibtex.bib}{}
\bibliographystyle{acm}
\vspace{1 pc}
\noindent  Aapo Kauranen \& Pekka Koskela\\
Department of Mathematics and Statistics\\
University of Jyv\"{a}skyl\"{a}\\
P.O. Box 35 (MaD)\\
FI-40014\\
Finland\\

\noindent{\it E-mail addresses}: \texttt{aapo.p.kauranen@jyu.fi}\\
\texttt{pkoskela@maths.jyu.fi}
\end{document}